\documentclass[12pt]{article}

\usepackage{amsthm} \usepackage{amsmath} \usepackage{amsfonts}
\usepackage{amssymb} \usepackage{latexsym} \usepackage{enumerate}
\usepackage{mathrsfs}
\usepackage[numbers]{natbib}
\usepackage{tikz}
\usepackage{pgfplots}

\RequirePackage[colorlinks,citecolor=blue,urlcolor=blue]{hyperref}

\theoremstyle{plain}%
\newtheorem{Theorem}{Theorem}[section]

\newtheorem{Lemma}[Theorem]{Lemma}
\newtheorem{Proposition}[Theorem]{Proposition}

\theoremstyle{definition}%

\newtheorem{Example}[Theorem]{Example}

\theoremstyle{remark}%

\newcommand{\abs}[1]{\left\lvert#1\right\rvert}

\newcommand{\braces}[1]{\left\lbrace{#1}\right\rbrace}
\newcommand{\cbraces}[1]{\left({#1}\right)} 

\newcommand*{\cpartial}{\overline{\partial}}

\newcommand{\descr}[2]{\left\lbrace{#1}|\; #2 \right\rbrace}

\newcommand{\set}{:=}
\newcommand{\ip}[2]{\left\langle #1,#2 \right\rangle}

\newcommand{\map}[3]{#1:\; #2 \rightarrow #3} 
\newcommand{\msym}[2]{\mathcal{S}^{#1}_{#2}} 
                                
\newcommand{\smap}[3]{#1:\; #2 \rightrightarrows #3}

\newcommand{\norm}[1]{\lVert#1\rVert}
\newcommand{\real}[1]{\mathbb{R}^{#1}}
\newcommand{\realext}{\mathbb{R} \cup \braces{+\infty}}

\newcommand{\rplus}[2]{\descr{x\in \real{#1}}{x^{#2}>0}}

\DeclareMathOperator*{\argmax}{arg\:max}
\DeclareMathOperator*{\argmin}{arg\:min}
\DeclareMathOperator{\Arg}{Arg}

\DeclareMathOperator*{\conv}{conv}

\DeclareMathOperator*{\cl}{cl}

\DeclareMathOperator{\dom}{dom}
\DeclareMathOperator{\diam}{diam}

\DeclareMathOperator{\epi}{epi}
\DeclareMathOperator*{\interior}{int}

\DeclareMathOperator*{\range}{range}

\DeclareMathOperator{\graph}{graph}

\newcommand{\envspace}{\vspace{2mm}}

\usepackage[color = orange!5]{todonotes}
\usepackage{enumitem}
\usepackage{varwidth}

\begin{document}

\title{Singularities of Fitzpatrick and convex functions}

\author{Dmitry Kramkov\footnote{Carnegie Mellon University, Department
    of Mathematical Sciences, 5000 Forbes Avenue, Pittsburgh, PA,
    15213-3890, USA,
    \href{mailto:kramkov@cmu.edu}{kramkov@cmu.edu}. The author also
    has a research position at the University of Oxford.}  and Mihai
  S\^{i}rbu\footnote{The University of Texas at Austin, Department of
    Mathematics, 2515 Speedway Stop C1200, Austin, Texas 78712,
    \href{mailto:sirbu@math.utexas.edu}{sirbu@math.utexas.edu}. The
    research of this author was supported in part by the National
    Science Foundation under Grant DMS 1908903.}  }

\date{\today}

\maketitle

\begin{abstract}
  In a pseudo-Euclidean space with scalar product $S(\cdot, \cdot)$,
  we show that the singularities of projections on $S$-monotone sets
  and of the associated Fitzpatrick functions are covered by countable
  $c-c$ surfaces having positive normal vectors with respect to the
  $S$-product.  By~\citet{Zajicek:79}, the singularities of a convex
  function $f$ can be covered by a countable collection of $c-c$
  surfaces.  We show that the normal vectors to these surfaces are
  restricted to the cone generated by $F-F$, where
  $F\set \cl \range \nabla f$, the closure of the range of the
  gradient of $f$.
\end{abstract}

\begin{description}
\item[Keywords:] convexity, subdifferential, Fitzpatrick function,
  projection, pseudo-Euclidean space, normal vector, singularity.
\item[AMS Subject Classification (2020):] 26B25,
  26B05,
  47H05,
  52A20.
\end{description}

\section{Introduction}

A locally Lipschitz function $f=f(x)$ on $\real{d}$ is differentiable
almost everywhere, according to the Rademacher's Theorem.  The set of
its singularities
\begin{displaymath}
  \Sigma(f) \set \descr{x\in \real{d}}{f \text{ is not differentiable at
    } x}
\end{displaymath}
can be quite irregular. For instance, for $d=1$, \citet{Zahorski:46}
(see \citet{FowlPreiss:09} for a simple proof) shows that \emph{any}
$G_{\delta \sigma}$ set (countable union of countable intersections of
open sets) of Lebesgue measure zero is the singular set of \emph{some}
Lipschitz function.

By \citet{Zajicek:79}, the $\Sigma(f)$ of a convex function $f=f(x)$
on $\real{d}$ has $c-c$ structure: it can be covered by a countable
collection of the graphs of the differences of two convex functions of
dimension $d-1$. Short proofs of the \citeauthor{Zajicek:79} theorem
can be found in~\citet[Theorem~4.20, p.~93]{BenLin:00},
\citet[Theorem~12.22, p.~1147]{Thib:23},
and~\citet{Hajlasz:22}. \citet{Alberti:94} shows that, except for sets
with zero $\mathcal{H}^{d-1}$ measure, the Hausdorff measure of
dimension $d-1$, the covering can be achieved with smooth
surfaces. These results yield sharp conditions for the existence and
uniqueness of optimal maps in $\mathcal{L}_2$ optimal transport,
see~\citet{Bren:91} and \citet[Theorem~1.26]{AmbrGigli:13}.

Let $E$ be a closed subset of $\real{d}$ and
\begin{displaymath}
  d_E(x)\set\inf _{y\in E}|x-y|, \quad 
  P_E(x)\set \descr{y\in E}{|x-y|=d_E(x)},
\end{displaymath}
be the Euclidean distance to $E$ and the projection on $E$,
respectively. \citet{Erdos:46} proves that the singular set
\begin{displaymath}
  \Sigma (P_E)
  \set \descr{x\in \real{d}}{P_E(x) \text{ contains at least two points}}
\end{displaymath}
can be covered by countable sets with finite $\mathcal {H}^{d-1}$
measure.  \citet{Hajlasz:22} uses \cite{Zajicek:79} and the
observation of \citet{Asplund:69} that the function
\begin{displaymath}
  \psi _E (x)\set \frac 12 |x|^2-\frac 12 d_E^2(x), \quad x\in \real{d},
\end{displaymath}
is convex, to conclude that $\Sigma (P_E)$ has the $c-c$ structure.
\citet{AlbanCannar:99} obtain a lower bound on the size of the set
$\Sigma (d_E)$, where $d_E$ is not differentiable.

Let $S$ be a $d\times d$ invertible symmetric matrix with
$m \in \braces{0,1,\dots,d}$ positive eigenvalues.  Motivated by
applications to backward martingale transport in \cite{KramXu:22},
\cite{KramSirb:22a}, and \cite{KramSirb:23}, we investigate in this
paper the singularities of projections on monotone sets in the
pseudo-Euclidean space with the scalar product
\begin{displaymath}
  S(x,y) \set \ip{x}{Sy} = \sum_{i,j=1}^d x^iS^{ij} y^j, \quad x,y\in 
  \real{d}. 
\end{displaymath}

Let $G\subset \real{d}$ be an $S$-monotone or $S$-positive set:
\begin{displaymath}
  S(x-y,x-y)\geq 0, \quad x,y \in G.
\end{displaymath}
For every $x\in \real{d}$, the scalar square to $G$ and the projection
on $G$ are given by
\begin{align*}
  \phi _G(x) & \set \inf _{y\in G}S(x-y,x-y), \\
  P_G(x) & \set \argmin_{y\in G}S(x-y,x-y).
\end{align*}
Note that $S$-monotonicity is equivalent to
$x\in G\implies x\in P_G(x)$ and that
\begin{displaymath}
  \psi _G(x)\set \frac 12 S(x,x)- \frac 12 \phi _G(x), \quad
  x\in\real{d}, 
\end{displaymath}
is the Fitzpatrick function studied in \cite{Fitzpatrick:88},
\cite{Simons:07}, and \cite{Penot:09}.

The singularities of the projection $P_G$ can be classified as
\begin{align*}
  \Sigma(P_G)  & \set  \descr{x\in \real{d}}{P_G(x) \text{ contains at
                 least two points}} \\
               &\; =  {\Sigma}_0(P_G) \cup \Sigma_1(P_G), \\
  { \Sigma}_0(P_G)  & \set \descr{x\in \Sigma(P_G)}{S(y_1-y_2,y_1-y_2)=0
                      \text{~for \emph{all}~} y_1,y_2 \in P_G(x) }, \\
  \Sigma_1(P_G)  & \set \descr{x\in \Sigma(P_G)}{S(y_1-y_2,y_1-y_2)>0
                   \text{~for \emph{some}~} y_1,y_2 \in P_G(x)}. 
\end{align*}
By Theorem \ref{th:4}, $\Sigma _1(P_G)$ is contained in a countable
union of $c-c$ surfaces having strictly positive normal vectors in the
$S$-space. The structure of the zero-order singularities is described
in Theorems~\ref{th:6} and~\ref{th:5}.  If $m=1$, then
${\Sigma}_0(P_G)$ is covered by a countable number of hyperplanes
having isotropic normal vectors in the $S$-space. If $m\geq 2$, then
${\Sigma}_0(P_G)$ is covered by a countable family of $c-c$ surfaces
whose normal vectors are positive and almost isotropic in the
$S$-space. These results yield sharp conditions for the existence and
uniqueness of backward martingale maps in~\cite{KramSirb:23}.

Using similar tools, in Theorem \ref{th:2}, we improve the $c-c$
description of singularities of general convex functions $f=f(x)$ from
\citet{Zajicek:79} by showing that the covering surfaces have normal
vectors belonging to the cone generated by $F-F$, where
$F\set \cl \range \nabla f$, the closure of the range of the gradient
of $f$.

\section{Parametrization of singularities}
\label{sec:param-sing}

We say that a function $g = g(x)$ on $\real{d}$ has \emph{linear
  growth} if there is a constant $K=K(g)>0$ such that
\begin{displaymath}
  \abs{g(x)} \leq K (1 + \abs{x}), \quad x\in \real{d}.  
\end{displaymath}
We write $\dom{\nabla g}$ for the set of points where $g$ is
differentiable.

Let $j\in \braces{1,\dots,d}$. We denote by $\mathcal{C}^j$ the
collection of compact sets $C$ in $\real{d}$ such that
\begin{displaymath}
  y^j=1, \quad y\in C. 
\end{displaymath}
Any compact set $C\subset \rplus{d}{j}$ can be rescaled as
\begin{displaymath}
  \theta^j(C)\set \descr{  \frac{y}{y^j}}{y\in C}\in \mathcal{C}^j.
\end{displaymath}
For $x \in \real{d}$, we denote by $x^{-j}$ its sub-vector without the
$j$th coordinate:
\begin{displaymath}
  x^{-j} \set 
  \cbraces{x^1,\dots,x^{j-1},x^{j+1},\dots,x^{d}}
  \in \real{d-1}. 
\end{displaymath}
For $C\in \mathcal{C}^j$, we write $\mathcal{H}^{j}_{C}$ for the
family of functions $h=h(x)$ on $\real{d}$ having the decomposition:
\begin{equation}
  \label{eq:1}
  h(x) = x^j + g_1(x^{-j}) - g_2(x^{-j}), \quad
  x\in \real{d},    
\end{equation}
where the functions $g_1$ and $g_2$ on $\real{d-1}$ are convex, have
linear growth, and
\begin{equation}
  \label{eq:2}
  \nabla h(x)\in C, \quad x^{-j} \in  \dom{\nabla g_1} \cap \dom{\nabla
    g_2}.   
\end{equation}

The latter property has a clear geometric interpretation. Let $H$ be a
closed set in $\real{d}$ and $x\in H$. Following \cite[Definition~6.3
on page~199]{RockWets:98}, we call a vector $w\in \real{d}$
\emph{regular normal to $H$ at $x$} if
\begin{displaymath}
  \limsup_{\substack{H\ni y\to x \\ y\not= x}} 
  \frac{\ip{w}{y-x}}{\abs{y-x}} \leq 0. 
\end{displaymath}
A vector $w\in \real{d}$ is called \emph{normal to $H$ at $x$} if
there exist $x_n\in H$ and a regular normal vector $w_n$ to $H$ at
$x_n$ such that $x_n\rightarrow x$ and $w_n\rightarrow w$.

For a set $B$ in $\real{d}$, we denote by $\conv{B}$ its convex hull.
\begin{Lemma}
  \label{lem:1}
  Let $j\in \braces{1,\dots,d}$, $C\in \mathcal{C}^j$, $h$ be given
  by~\eqref{eq:1} for convex functions $g_1$ and $g_2$ with linear
  growth, and $H$ be the zero-level set of $h$:
  \begin{displaymath}
    H \set \descr{x\in \real{d}}{h(x) = 0}.
  \end{displaymath}
  Then $h\in \mathcal{H}^j_C$, that is, \eqref{eq:2} holds, if and
  only if for every $x\in H$, there exists a normal vector $w\in C$ to
  $H$ at $x$.
\end{Lemma}

\begin{proof}
  We can assume that $j=1$. Denote $f\set g_1-g_2$, so that
  $h(x)=x^1+f(x^{-1})$.  We have that
  $\dom {\nabla{h}}=\real{}\times \dom{\nabla{f}}$,
  $\nabla h(x)= \cbraces{1, \nabla f(x^{-1})}$, and
  \begin{displaymath}
    H=\descr{(-f(u), u)}{u\in \real{d-1}}.
  \end{displaymath}
  Denote also
  $U\set \dom{\nabla g_1} \cap \dom{\nabla g_2} \subset \dom
  {\nabla{f}}$.
  
  $\implies$: Clearly, the gradient $\nabla h(x)$ is a regular normal
  vector to $H$ at $x\in \dom{\nabla h}\cap H$. As $U$ is dense in
  $\real{d-1}$, the result holds by standard compactness arguments.

  $\impliedby$: Let $u\in U$, $v\in \real{d-1}$, and $w=(1,v)\in C$ be
  a normal vector to $H$ at $x=(-f(u),u)$.  Take a sequence
  $(x_n,w_n)$, $n\geq 1$, that converges to $(x,w)$ and where $w_n$ is
  a regular normal vector to $H$ at $x_n\in H$ with $w_n^1=1$. Such a
  sequence exists by the definition of a normal vector. We can
  represent $x_n=(-f(u_n), u_n)$ for $u_n\in \real{d-1}$ and
  $w_n=(1, v_n)$ for $v_n\in \real{d-1}$.
  
  We claim that $v_n$ belongs to the Clarke gradient of $f$ at $u_n$:
  \begin{displaymath}
    v_n \in \cpartial f(u_n) \set  \conv \descr{\lim_m \nabla
      f(r_m)}{\dom \nabla f \ni r_m \to u_n}.
  \end{displaymath}
  The continuity of $\nabla f$ at $u\in U$ then yields that
  $v_n \to \nabla f(u)$. Hence, $\nabla f(u)=v$, implying that
  $\nabla h(x) =(1, \nabla f(u))= w \in C$.
  
  In order to prove the claim, we write the definition of
  $w_n=(1,v_n)$ being regular normal to $H$ at $x_n=(-f(u_n), u_n)$ as
  \begin{displaymath}
    \limsup _{\substack{r\rightarrow  u_n \\ r\not= u_n}}
    \frac{-\cbraces{f(r)-f(u_n)}+\ip{v_n}{r-u_n}}{|f(r)-f(u_n)|+|r-u_n|}
    \leq 0. 
  \end{displaymath}
  Using the Lipschitz property of $f$, we obtain that
  \begin{displaymath}
    \ip{v_n}{s}\leq \limsup_{\delta \downarrow 0}\frac{f(u_n+\delta
      s)-f(u_n)}{\delta}, \quad s\in \real{d-1}. 
  \end{displaymath}
  By~\cite[Corollary~1.10]{Clarke:75}, we have that
  $v_n\in \cpartial f (u_n)$, as claimed.
\end{proof}

We recall that the subdifferential
$\smap{\partial f}{\real{d}}{\real{d}}$ of a closed convex function
$\map{f}{\real{d}}{\realext}$ is defined as
\begin{displaymath}
  \partial f(x)\set \descr{y\in \real d}{\ip{z}{y}\leq f(x+z)-f(x), \,
    z\in \real{d}}. 
\end{displaymath}
Clearly,
$\dom \partial f \set \descr{x\in \real{d}}{\partial f(x)\not=
  \emptyset} \subset \dom f \set \descr{x\in \real{d}}{f(x)<\infty}$.

The following theorem is our main technical tool for the study of
singularities of convex and Fitzpatrick functions in
Sections~\ref{sec:sing-points-conv} and~\ref{sec:sing-points-fitzp}.

\begin{Theorem}
  \label{th:1}
  Let $\map{f}{\real{d}}{\realext}$ be a closed convex function,
  $j\in \braces{1,\dots, d}$, and $C_1$ and $C_2$ be compact sets in
  $\real{d}$ such that
  \begin{displaymath}
    y^j> 0, \quad y \in C_2-C_1. 
  \end{displaymath}
  There exists a function $h\in \mathcal {H}^j_{\theta^j(C_2-C_1)}$
  such that
  \begin{align*}
    \Sigma_{C_1,C_2}(\partial f)
    \set & \descr{x\in \dom{\partial f(x)}}{\partial f(x) \cap C_i \not
           = \emptyset, \, i=1,2} \\ 
    \subset & \descr{x\in \real{d}}{h(x)=0}.   
  \end{align*}
\end{Theorem}

The proof of the theorem relies on Lemma~\ref{lem:2}.  For a closed
set $A\subset \real{d}$, we denote
\begin{equation}
  \label{eq:3}
  f_A(x): = \sup_{y\in A}\cbraces{\ip{x}{y} - f^*(y)}, \quad x\in
  \real{d},
\end{equation}
where $f^*$ is the convex conjugate of $f$:
\begin{displaymath}
  f^*(y)\set \sup_{x\in \real{d}}\cbraces{\ip{x}{y}-f(x)} \in
  \realext, \quad y \in 
  \real{d}.
\end{displaymath}
We have that $f_A$ is a closed convex function taking values in
$\realext$ if and only if
\begin{displaymath}
  A\cap \dom {f^*} = \descr{x\in A}{f^*(x)< \infty} \not= \emptyset;
\end{displaymath}
otherwise, $f_A=-\infty$. We recall that
\begin{equation}
  \label{eq:4}
  f(x)=\ip{x}{y}-f^*(y)\iff y\in 
  \partial f(x) \iff x\in \partial f^*(y). 
\end{equation}

\begin{Lemma}
  \label{lem:2}
  Let $\map{f}{\real{d}}{\realext}$ be a closed convex function and
  $C$ be a compact set in $\real{d}$ such that
  $C\cap \dom{f^*}\not=\emptyset$.  Then $f_C$ has linear growth and
  for every $x\in \real{d}$,
  \begin{gather*}
    \partial f_C(x) \cap C = \Arg_C(x) \set \argmax_{y\in
      C}\cbraces{\ip{x}{y} -
      f^*(y)} \not=\emptyset, \\
    \partial f_C(x)  = \conv{\cbraces{\partial f_C(x) \cap C}}, \\
    \partial f(x) \cap C \not= \emptyset \iff f(x) = f_C(x) \iff
    \partial f(x) \cap C = \partial f_C(x) \cap C.
  \end{gather*}
  In particular, $f_C$ is differentiable at $x$ if and only if
  $\partial f_C(x) \cap C$ is a singleton, in which case
  \begin{displaymath}
    \partial f_C(x) = \braces{\nabla f_C(x)} \in C. 
  \end{displaymath}
\end{Lemma}

\begin{proof}
  Since $C$ is a compact set, $f^*$ is a closed convex function, and
  $C\cap\dom {f^*}\not= \emptyset$, we have that
  \begin{displaymath}
    \sup_{y\in C} \abs{y} < \infty, \quad \inf_{x\in C}
    f^*(x) > -\infty,  
  \end{displaymath}
  and that $\Arg_C(x)$ is a non-empty compact.  Let
  $y_0\in C\cap \dom {f^*}$.  From the definition of $f_C$ we deduce
  that
  \begin{displaymath}
    -\abs{x} \abs{y_0} -f^*(y_0) \leq f_C(x) \leq \abs{x}\sup_{y\in
      C}\abs{y}-\inf_{y\in C}f^*(y), \quad x\in \real{d}.  
  \end{displaymath}
  It follows that $f_C$ has linear growth.
  
  The function
  \begin{displaymath}
    h(x,y) \set \ip{x}{y}-f^*(y) \in \real{}\cup\braces{-\infty}, \quad
    x,y\in \real{d}, 
  \end{displaymath}
  is linear in $x$ and concave and upper semi-continuous in $y$. Fix
  $x\in \real{d}$.  We can choose $K$ large enough such that for
  \begin{displaymath}
    E\set \descr{z\in C}{f^*(z)\leq K},
  \end{displaymath}
  we have that $f_C=f_{E}$ in a neighborhood of $x$ and
  \begin{displaymath}
    \Arg _C(x)=\Arg _E(x)\set
    \argmax_{y\in E}\cbraces{\ip{x}{y} -f^*(y)}.
  \end{displaymath}
  Since $E$ is compact and the function $h(\cdot, y)$ is finite for
  $y\in E$, the classical envelope theorem \cite[Theorem~4.4.2,
  p.~189]{HiriarLemar:01} yields that
  \begin{align*}
    \partial f_C(x) \
    & =  \partial f_E(x)=  \partial \max_{y\in E}h(x,y) =
      \conv\bigcup_{y\in \Arg_E(x)} \partial_x h(x,y) =  \conv
      \Arg_E(x) \\
    & = \conv{\Arg _C(x)}.
  \end{align*}
  From the concavity of $h(x,\cdot)$ we deduce that
  \begin{displaymath}
    \partial f_C(x) \cap C= \cbraces{\conv{\Arg _C(x)}}  \cap C
    = \Arg _C(x).  
  \end{displaymath}

  Clearly, $f_C\leq f$. If $y\in \partial f(x) \cap C$, then
  $f(x) = \ip{x}{y} - f^*(y) \leq f_C(x)$.  Hence, $f_C(x)= f(x)$ and
  $y\in \Arg _C(x)\subset \partial f_C(x)$.

  Conversely, let $f_C(x) = f(x)$. For every
  $y\in \partial f_C(x) \cap C$, we have that
  $f(x) = f_C(x) = \ip{x}{y} - f^*(y)$ and then that
  $y\in \partial f(x)\cap C$.
  
  Finally, being a convex function, $f_C$ is differentiable at $x$ if
  and and only if $\partial f_C(x)$ is a singleton. In this case,
  $\partial f_C(x) = \braces{\nabla f_C(x)}$.
\end{proof}

\begin{proof}[Proof of Theorem~\ref{th:1}]
  Hereafter, $i=1,2$. We assume that
  $\Sigma_{C_1,C_2}(\partial f) \not= \emptyset$ as otherwise, there
  is nothing to prove. This implies that
  $C_i\cap \dom {f^*}\not=\emptyset$.  Let $C\set C_1\cup C_2$.
  Lemma~\ref{lem:2} yields that
  \begin{equation}
    \label{eq:5}
    \Sigma_{C_1,C_2}(\partial f) = \Sigma_{C_1,C_2}(\partial
    f_C) \cap \descr{x\in \real{d}}{f(x) = f_C(x)}. 
  \end{equation}

  To simplify notations we assume that $j=1$.  Denote by
  \begin{displaymath}
    a \set \max_{y\in C_1} {y}^1 <\min_{y \in C_2} {y}^1 \set b. 
  \end{displaymath}  
  We write $x \in \real{d}$ as $(t,u)$, where $t\in \real{}$ and
  $u\in \real{d-1}$, and define the saddle function
  \begin{displaymath}
    g(t,u): = \inf_{s\in \real{}} \cbraces{f_C(s,u) - st},
    \quad a<t<b, u\in \real{d-1}. 
  \end{displaymath}  

  Select $y_i =(q_i, z_i)\in C_i\cap \dom {f^*}$. We have that
  \begin{align*}
    f_C(s,u) & \geq \max_{i=1,2} \cbraces{sq_i+\ip{u}{z_i} - f^*(y_i)}.
  \end{align*}
  Since $q_1\leq a<b\leq q_2$, it follows from the definition of $g$
  that
  \begin{displaymath}
    -\infty< g(t,u) \leq f_C(0,u), \quad a<t<b, u\in \real{d-1}.
  \end{displaymath}
  By Lemma~\ref{lem:2}, $f_C$ has linear growth. The classical results
  on saddle functions, Theorems~33.1 and~37.5 in \cite{Rock:70}, imply
  that
  \begin{enumerate}[label = {\rm (\roman{*})}, ref={\rm (\roman{*})}]
  \item \label{item:1} For every $a<t <b$, the function $g(t,\cdot)$
    is convex and has linear growth on $\real{d-1}$.
  \item \label{item:2} For every $u\in \real{d-1}$, the function
    $g(\cdot, u)$ is concave and finite on $(a,b)$.
  \item \label{item:3} For any $a<t<b$, we have that
    $(t,v) \in \partial f_C(s,u)$ if and only if
    $v \in \partial_u g(t,u)$ and $-s \in \partial_tg(t,u)$. In this
    case, $g(t,u) = f_C(s,u) - st$.
  \end{enumerate}
  
  We take $a <r_1<r_2<b$ and denote $g_i \set g(r_i, \cdot)$. We
  verify the assertions of the theorem for the function
  \begin{displaymath}
    h(s,u)\set  s +
    \frac1{r_2-r_1}\cbraces{g_2(u) - g_1(u)}, \quad s\in \real{}, u\in
    \real{d}.  
  \end{displaymath} 
  More precisely, we show that
  \begin{displaymath}
    \Sigma_{C_1,C_2}(\partial f_C)= \descr{x\in \real{d}}{h(x)=0},   
  \end{displaymath}
  which together with~\eqref{eq:5} implies the result.

  Let $x=(s,u)$ and $(t_i,v_i) \in \partial f_C(x) \cap C_i$.  As
  $t_1 \leq a<r_i<b \leq t_2$, the convexity of subdifferentials
  yields that $(r_i,w_i) \in \partial f_C(x)$, where
  \begin{equation}
    \label{eq:6}
    w_i = \frac{t_2-r_i}{t_2-t_1} v_1 + \frac{r_i-t_1}{t_2-t_1} v_2.
  \end{equation}
  By~\ref{item:3}, $g_i(u) = f_C(x) - s r_i$ and then $h(x) = 0$.

  Conversely, let $x=(s,u)$ be such that $h(x) = 0$ or, equivalently,
  \begin{displaymath}
    -s = \frac{g_2(u) - g_1(u)}{r_2-r_1} = \frac{g(r_2,u) - g(r_1,u)}{r_2-r_1}.
  \end{displaymath}
  The mean-value theorem yields $r\in [r_1,r_2]$ such that
  $-s \in \partial_t g(r,u)$.  Observe that $a<r<b$. Taking any
  $w \in \partial_u g(r,u)$, we deduce from~\ref{item:3} that
  $y\set (r,w) \in \partial f_C(x)$. As
  $\partial f_C = \conv \cbraces{\partial f_C \cap C}$, the point $y$
  is a convex combination of some $y_i \in \partial f_C(x) \cap C_i$,
  $i=1,2$.  In particular, $x\in \Sigma_{C_1,C_2}(\partial f_C)$.

  Finally, let $x=(s,u)$ be such that $h(x) = 0$ and $g_1$ and $g_2$
  are differentiable at $u$. As we have already shown, the gradients
  $w_i \set \nabla g_i(u)$ are given by~\eqref{eq:6} for some
  $(t_i,v_i) \in \partial f_C(x) \cap C_i$. It follows that
  \begin{displaymath}
    \nabla h(x) = \cbraces{1, \frac{w_2-w_1}{r_2-r_1}} = \cbraces{1,
      \frac{v_2-v_1}{t_2-t_1}} \in 
    \theta^1(C_2-C_1). 
  \end{displaymath}
  Hence, $h \in \mathcal{H}^1_{\theta^1(C_2-C_1)}$.
\end{proof}

\section{Singular points of convex functions}
\label{sec:sing-points-conv}

For a multi-function $\smap{\Psi}{\real{d}}{\real{d}}$ taking values
in closed subsets of $\real{d}$, we denote its domain by
\begin{displaymath}
  \dom \Psi \set \descr{x\in  \real{d}}{\Psi (x)\not= \emptyset}.
\end{displaymath}
Given an index $j\in \braces{1,\dots,d}$ and a closed set $A$ in
$\real{d}$, the \emph{singular} set of $\Psi$ is defined as
\begin{gather*}
  \Sigma^j_{A}(\Psi) \set \descr{x\in \dom \Psi}{\exists y_1,y_2 \in
    \Psi(x) \cap A \text{ with } y_1^j\not=y_2^j}.
\end{gather*}
We also write $\Sigma^j(\Psi) = \Sigma ^j_{\real{d}}(\Psi)$ and
$\Sigma(\Psi) = \cup_{j=1,\dots,d} \Sigma^j(\Psi)$.

Let $\map{f}{\real{d}}{\realext}$ be a closed convex function such
that its domain has a non-empty interior:
\begin{displaymath}
  D\set \interior\dom f \not= \emptyset. 
\end{displaymath}
It is well-known that
$\dom \nabla f \set \descr{x\in D}{\nabla f(x) \text{~exists}}$ is
dense in $D$ and
\begin{displaymath}
  D \setminus \dom \nabla f = \Sigma(\partial f) \cap D =
  \descr{x\in D}{\partial f(x) \text{ is not a point}}.  
\end{displaymath}
According to~\cite{Zajicek:79}, see also~\cite{Alberti:94}
and~\cite{Hajlasz:22}, this set of \emph{interior} singularities can
be covered by countable $c-c$ surfaces
$H_n = \descr{x\in \real{d}}{h_n(x) = 0}$, $n\geq 1$, where
\begin{displaymath}
  h_n(x) = x^j + g_{n,1}(x^{-j}) - g_{n,2}(x^{-j}), \quad
  x\in \real{d},    
\end{displaymath}
for some $j\in \braces{1,\dots,d}$ and finite convex functions
$g_{n,1}$ and $g_{n,2}$ on $\real{d-1}$.

Theorem~\ref{th:2} and Lemma \ref{lem:1} describe the
\emph{orientation} of the covering surface $H_n$ by showing that, at
any point, it has a normal vector $w$ with $w^j=1$, that belongs to
the cone $K$ generated by $F-F$, where
\begin{displaymath}
  F\set \cl \range \nabla f = 
  \descr{\lim_n \nabla f(x_n)}{x_n \in \dom
    \nabla f}.  
\end{displaymath}
Of course, this information has some value only if $K$ is
distinctively smaller than $\real{d}$. This is the case for
Fitzpatrick functions in the pseudo-Euclidean space $S$, where
according to~Theorem~\ref{th:3}, $K$ contains only $S$-non-negative
vectors. Proposition~\ref{prop:1} provides the geometric
interpretation of the $\range \nabla f$ and Lemma~\ref{lem:4} explains
the special feature of its closure $F$.

It turns out that the same surfaces $(H_n)$ also cover the
singularities of the Clarke-type subdifferential $\cpartial f$ defined
on the \emph{closure} of $D$:
\begin{displaymath}
  \cpartial f(x) \set  \cl \conv \descr{\lim_n \nabla
    f(x_n)}{\dom
    \nabla f \ni x_n \to x}, \quad x\in \cl D.     
\end{displaymath}
By Theorem~25.6 in~\cite{Rock:70},
\begin{displaymath}
  \partial f(x) = \cpartial f(x) + N_{\cl D}(x), \quad x\in \cl D, 
\end{displaymath}
where $N_A(x)$ denotes the normal cone to a closed convex set
$A\subset \real{d}$ at $x\in A$:
\begin{align*}
  N_{A}(x) \set & \descr{s\in \real{d}}{\ip{s}{y-x}\leq 0 \text{ for
                  all } y\in A} \\
  = & \descr{s\in \real{d}}{s \text{ is a regular normal vector to } A \text{ at }
      x}. 
\end{align*}
Recalling that $0\in N_{A}(x)$ for $x\in A$ and $N_{A}(x) =\braces{0}$
for $x\in \interior{A}$, we deduce that
\begin{align}
  \label{eq:7}
  \dom{\cpartial f} & = \dom{\partial f}, \\
  \label{eq:8}
  \cpartial f(x) & \subset \partial f(x), \quad x \in \cl D,  \\
  \label{eq:9}
  \cpartial f(x) & = \partial f(x), \quad x\in D.
\end{align}

The diameter of a set $E$ is denoted by
$\diam{E} \set \sup_{x,y\in E} \abs{x-y}$.

\begin{Theorem}
  \label{th:2}
  Let $\map{f}{\real{d}}{\realext}$ be a closed convex function with
  $D\set \interior\dom f \not= \emptyset$. Let
  $j\in \braces{1,\dots,d}$ and $A$ be a closed set in $\real{d}$
  containing $\range \nabla f$.  Then
  \begin{gather}
    \label{eq:10}
    \Sigma^j(\partial f) \cap D = \Sigma^j(\cpartial f) \cap D =
    \Sigma^j_A(\partial f)\cap
    D, \\
    \label{eq:11}
    \Sigma^j(\cpartial f) \subset \Sigma^j_A(\partial f).
  \end{gather}
  If $y^j = z^j$ for all $y,z \in A$, then, clearly,
  $\Sigma_A^j(\partial f)=\emptyset$. Otherwise, for every $n\geq 1$,
  there exist a compact set $C_n \subset A-A$ with $y^j > 0$,
  $y\in C_n$, and a function $h_n \in \mathcal{H}^j_{\theta^j(C_n)}$
  such that
  \begin{equation}
    \label{eq:12}
    \Sigma_A^j(\partial f) \subset
    \bigcup_n \descr{x\in \real{d}}{h_n(x)=0}.  
  \end{equation}
  For any $\epsilon>0$, all $C_n$ can be chosen so that
  $\diam{\theta^j(C_n)}< \epsilon$.
\end{Theorem}

Taking the unions over $j\in \braces{1,\dots,d}$, we obtain the
descriptions of the \emph{full} singular sets $\Sigma(\partial f)$ on
$D$, $\Sigma(\cpartial f)$, and $\Sigma_A(\partial f)$. Taking a
smaller $\epsilon>0$ in the last sentence of the theorem, we make the
directions of the normal vectors to the covering surface $H_n$ closer
to each other. As a result, $H_n$ gets approximated by a hyperplane.

Theorem~\ref{th:2} uses a larger closed set $A$ instead of $F$ to
allow for more flexibility in the treatment of singularities of
$\partial f$ on the boundary of $D$.  Lemma~\ref{lem:5} shows that
\begin{displaymath}
  \partial f(x) \cap A
  = \Arg_A(x) \set \argmax_{y\in A}\cbraces{\ip{x}{y} -
    f^*(y)}, \quad x \in \cl D. 
\end{displaymath}
In the framework of Fitzpatrick functions in
Section~\ref{sec:sing-points-fitzp}, $\Arg_A$ becomes a projection on
the monotone set $A$ in a pseudo-Euclidean space.

The proof of Theorem~\ref{th:2} relies on some lemmas. We start with a
simple fact from convex analysis.

\begin{Lemma}
  \label{lem:3}
  Let $\map{f}{\real{d}}{\realext}$ be a closed convex function
  attaining a strict minimum at a point $x_0$:
  \begin{displaymath}
    f(x_0)<f(x), \quad x\in \real{d}, \, x\not =x_0.
  \end{displaymath}
  Then
  \begin{displaymath}
    f(x_0) < \inf_{\abs{x-x_0} \geq \epsilon} f(x), \quad \epsilon >
    0.  
  \end{displaymath}
\end{Lemma}

\begin{proof}
  If the conclusion is not true, then there exist $\epsilon > 0$ and a
  sequence $(x_n)$ with $|x_n-x_0|\geq \epsilon$ such that
  $f(x_n)\rightarrow f(x_0)$.  As
  \begin{displaymath}
    z_n\set x_0+\frac{\epsilon}{\abs{x_n-x_0}}(x_n-x_0) 
  \end{displaymath}
  is a convex combination of $x_0$ and $x_n$, we deduce that
  \begin{displaymath}
    f(z_n) \leq \max(f(x_0),f(x_n)) = f(x_n) \to f(x_0). 
  \end{displaymath}
  By compactness, $z_n\rightarrow z_0$ over a subsequence. Clearly,
  $\abs{z_0-x_0}=\epsilon$, while by the lower semi-continuity,
  $f(z_0)\leq \liminf f(z_n)\leq f(x_0)$. We have arrived to a
  contradiction.
\end{proof}

The following result explains the special role played by
$\cl\range \nabla f$.  Recall the notation $f_A$ from~\eqref{eq:3}.

\begin{Lemma}
  \label{lem:4}
  Let $\map{f}{\real{d}}{\realext}$ be a closed convex function with
  $D\set \interior \dom f \not=\emptyset$ and $A$ be a closed set in
  $\real{d}$.  Then
  \begin{displaymath}
    f_A(x) = f(x), \, x\in \cl D  \iff  \range \nabla f \subset A. 
  \end{displaymath}
  In other words, $F\set \cl\range \nabla f$ is the minimal closed set
  such that $f_F=f$ on $\cl D$.
\end{Lemma}

\begin{proof}
  $\impliedby$: If $x \in \dom \nabla f$, then
  $y\set \nabla f(x)\in A$ and~\eqref{eq:4} yields that
  \begin{displaymath}
    f(x) = \ip{x}{y} - f^*(y) = f_A(x). 
  \end{displaymath}
  Since $\dom{\nabla f}$ is dense in $D$, the closed convex functions
  $f_A$ and $f$ coincide on $\cl D$.
 
  $\implies$: We fix $x_0\in \dom \nabla f$ and set
  $y_0\set \nabla f(x_0)$.  By the assumption, $f_A(x_0)=f(x_0)$.  If
  $y_0\notin A$, then the distance between $y_0$ and $A$ is at least
  $\epsilon >0$.  According to~\eqref{eq:4}, the concave upper
  semi-continuous function
  \begin{displaymath}
    y\rightarrow \ip{x_0}{y}-f^*(y)
  \end{displaymath}
  attains a strict global maximum at $y_0$ and has the maximum value
  $f(x_0)$.  By Lemma~\ref{lem:3},
  \begin{displaymath}
    \sup_{\abs{y-y_0} \geq  \epsilon}
    \cbraces{\ip{x_0}{y} - f^*(y)} <f(x_0).
  \end{displaymath} 
  As $A\subset \descr{y\in \real{d}}{|y-y_0|\geq \epsilon}$, we arrive
  to a contradiction:
  \begin{displaymath}
    f_A(x_0)=\sup_{y\in A}\cbraces{\ip{x_0}{y} - f^*(y)}    \leq
    \sup_{\abs{y-y_0}\geq  \epsilon} 
    \cbraces{\ip{x_0}{y} - f^*(y)} <f(x_0).
  \end{displaymath}
  Hence, $y_0 \in A$, as required.
\end{proof}

\begin{Lemma}
  \label{lem:5}
  Let $\map{f}{\real{d}}{\realext}$ be a closed convex function with
  $D\set \interior \dom f \not=\emptyset$. Let $A$ be a closed set in
  $\real{d}$ such that $f=f_A$ on $\cl D$.  Then
  \begin{gather*}
    \partial f(x) \cap A = \Arg_A(x) \set \argmax_{y\in
      A}\cbraces{\ip{x}{y} -
      f^*(y)}, \quad x \in \cl D, \\
    \dom \Arg_A \cap \cl D
    = \dom \partial f=\dom \cpartial f, \\
    \cpartial f(x) = \partial f(x) = \conv{\cbraces{\partial f(x) \cap
        A}}, \quad
    x\in D, \\
    \cpartial f(x) \subset \cl\conv\cbraces{\partial f(x) \cap A},
    \quad x\in \dom \cpartial f.
  \end{gather*}
\end{Lemma}
\begin{proof}
  If $x \in \dom \Arg_A \cap \cl D$ and $y\in \Arg_A(x)$, then
  \begin{equation}
    \label{eq:13}
    f(x) = f_A(x) = \ip{x}{y} - f^*(y).
  \end{equation}
  Hence, $y\in \partial f(x) \cap A$, by the properties of
  subdifferentials in~\eqref{eq:4}.

  Conversely, let $x\in \dom \partial f$.  Lemma~\ref{lem:4} shows
  that $F\set \cl\range \nabla f \subset A$. Accounting
  for~\eqref{eq:7} and~\eqref{eq:8} and the definition of
  $\cpartial f(x)$, we obtain that
  \begin{displaymath}
    \partial f(x) \cap A \supset \partial f(x) \cap F
    \supset \cpartial f(x)\cap F\not=\emptyset. 
  \end{displaymath}
  Let $y\in \partial f(x) \cap A$.  From~\eqref{eq:4} we
  deduce~\eqref{eq:13} and then that $y \in \Arg_A(x)$. We have proved
  the first two assertions of the lemma.

  The fact that $\cpartial f = \partial f$ on $D$ has been already
  stated in~\eqref{eq:9}. To prove the second equality in the third
  assertion, we fix $x_0\in D$ and choose $\epsilon>0$ such that
  \begin{displaymath}
    B(x_0,\epsilon) \set \descr{x\in \real{d}}{\abs{x-x_0}\leq
      \epsilon} \subset D. 
  \end{displaymath}
  The uniform boundedness of $\partial f$ on compacts in $D$ implies
  the existence of a constant $K>0$ such that
  \begin{displaymath}
    \abs{y} \leq K<\infty, \quad y\in \partial f(x), \, x\in
    B(x_0,\epsilon). 
  \end{displaymath}
  Denote by $A_K \set A \cap \descr{x\in \real{d}}{\abs{x}\leq K}$.
  As $\range \nabla f \subset A$, we deduce from~\eqref{eq:4} that
  \begin{displaymath}
    f(x)=f_{A_K}(x), \quad x\in B(x_0,\epsilon)\cap \dom \nabla f, 
  \end{displaymath}
  and then, by the density of $\dom \nabla f$ in $D$, that $f=f_{A_K}$
  on $B(x_0,\epsilon)$.  Finally, Lemma~\ref{lem:2} shows that
  \begin{displaymath}
    \partial f(x_0) = \partial f_{A_K}(x_0) = \conv\cbraces{\partial
      f_{A_K}(x_0) \cap A_K} = \conv\cbraces{\partial f(x_0) \cap A}. 
  \end{displaymath}

  If $\dom \nabla f \ni x_n \to x$ and $\nabla f(x_n) \to y$, then,
  clearly, $y\in F \subset A$. Moreover, $y \in \partial f(x)$, by the
  continuity of subdifferentials. The last assertion of the lemma
  readily follows.
\end{proof}

We are ready to finish the proof of Theorem~\ref{th:2}.

\begin{proof}[Proof of Theorem~\ref{th:2}]
  Lemma~\ref{lem:4} shows that $f=f_A$ on $\cl D$. Then, by
  Lemma~\ref{lem:5},
  \begin{gather*}
    \cpartial f(x) = \partial f(x) = \conv{\cbraces{\partial f(x) \cap
        A}}, \quad
    x\in D, \\
    \cpartial f(x) \subset \cl\conv\cbraces{\partial f(x) \cap A},
    \quad x\in \dom \cpartial f.
  \end{gather*}
  Recalling that $\dom{\partial f} = \dom{\cpartial f}$, we
  deduce~\eqref{eq:10} and~\eqref{eq:11}.
  
  Fix $\epsilon >0$.  Let $(x_n)$ be a dense sequence in $A$ and
  $(r_n)$ be an enumeration of all positive rationals.  Denote by
  $\alpha\set (m, n,k,l)$ the indexes for which the compacts
  \begin{displaymath}
    C^{\alpha}_1 \set\descr{x\in A }{ |x-x_m|\leq r_k}, \quad
    C^{\alpha}_2 \set\descr{x\in A}{|x-x_n|\leq r_l},
  \end{displaymath}
  satisfy the constraints:
  \begin{displaymath}
    \diam  \theta^j(C^{\alpha}_2-C^{\alpha}_1) < \epsilon \text{~and~}
    x^j > 0, \; x \in C_2^{\alpha} -
    C_1^{\alpha}.
  \end{displaymath}  
  We have that
  \begin{displaymath}
    \Sigma^j_A(\partial f) = \bigcup_{\alpha} \Sigma_{C^{\alpha}_1,
      C^{\alpha}_2}(\partial f) = \bigcup_{\alpha}  \descr{x\in
      \real{d}}{\partial f(x) \cap C^{\alpha}_i \not  
      = \emptyset, \, i=1,2}.
  \end{displaymath}
  For every index $\alpha$, Theorem~\ref{th:1} yields a function
  $h\in \mathcal H^j_{\theta^j(C^{\alpha}_2 - C^{\alpha}_1)}$ such
  that
  \begin{displaymath}
    \Sigma_{C^{\alpha}_1, C^{\alpha}_2}(\partial f)\subset \descr{x\in
      \real{d}}{h(x)=0}. 
  \end{displaymath}
  We have proved~\eqref{eq:12} and with it the theorem.
\end{proof}

We conclude the section with the geometric interpretation of
$\range \nabla f$. For a closed convex function
$\map{f}{\real{d}}{\realext}$, we denote by $\epi f$ its epigraph:
\begin{displaymath}
  \epi f \set \descr{(x,q) \in \real{d}\times \real{}}{f(x) \leq q}. 
\end{displaymath}
Let $E$ be a closed convex set in $\real{d}$. A point $x_0\in E$ is
called \emph{exposed} if there is a hyperplane intersecting $E$ only
at $x_0$. In other words, there is $y_0 \in \real{d}$ such that
\begin{displaymath}
  \ip{x-x_0}{y_0} > 0, \quad x\in E \setminus \braces{x_0}. 
\end{displaymath}

\begin{Proposition}
  \label{prop:1}
  Let $\map{f}{\real{d}}{\realext}$ be a closed convex function with
  $\interior{\dom f} \not= \emptyset$.  Then
  \begin{displaymath}
    \descr{\cbraces{y, f^*(y)} }{y\in \range \nabla f} = \text{
      exposed points of } \epi f^*. 
  \end{displaymath}
\end{Proposition}

\begin{proof}
  By definition, $(y_0,r_0)$ is an exposed point of $\epi f^*$ if it
  belongs to $\epi f^*$ and
  \begin{displaymath}
    \ip{y-y_0}{x_0} + (r-r_0) q_0 > 0, \quad (y,r) \in \epi f^*
    \setminus \braces{(y_0,r_0)}, 
  \end{displaymath}
  for some $(x_0,q_0)\in \real{d}\times \real{}$.  The definition of
  $\epi f^*$ ensures that $q_0>0$ and $r_0=f^*(y_0)$. Rescaling
  $(x_0,q_0)$ so that $q_0=1$, we deduce that the function
  \begin{displaymath}
    {y} \to {\ip{x_0}{y} +f^*(y)}
  \end{displaymath}
  has the unique minimizer $y_0$.  This is equivalent to $y_0$ being
  the only element of $\partial f(-x_0)$, which in turn is equivalent
  to $-x_0\in \dom \nabla f$ and $y_0=\nabla f(-x_0)$.
\end{proof}

\section{Singular points of projections in $S$-spaces}
\label{sec:sing-points-fitzp}

We denote by $\msym{d}{m}$ the family of symmetric
$d\times d$-matrices of full rank with $m \in \braces{0,1,\dots,d}$
positive eigenvalues. For $S \in \msym{d}{m}$, the bilinear form
\begin{displaymath}
  S(x,y) \set \ip{x}{Sy}  = \sum_{i,j=1}^d x^i S_{ij}y^j, \quad
  x,y\in \real{d},
\end{displaymath}
defines the scalar product on a pseudo-Euclidean space $\real{d}_m$
with dimension $d$ and index $m$, which we call the $S$-space. The
quadratic form $S(x,x)$ is the \emph{scalar square} on the $S$-space;
its value may be negative.

For a closed set $G\subset \real{d}$, we define the
\emph{Fitzpatrick-type} function
\begin{displaymath}
  \psi_G (x)\set  \sup _{y\in G}\cbraces{S(x,y)-\frac 12S(y,y)} \in
  \realext, 
  \quad x \in \real{d}, 
\end{displaymath}
and the projection multi-function
\begin{align*}
  P_G (x)\set & \argmin_{y\in G}S(x-y, x-y) = \argmax_{y\in G}\cbraces{S(x,y)-\frac 12S(y,y)}, \quad x \in \real{d}.
\end{align*}
Clearly, $\psi_G$ is a closed convex function and $P_G$ takes values
in the closed (possibly empty) subsets of $G$.

A closed set $G\subset \real{d}$ is called $S$-\emph{monotone} or
$S$-\emph{positive} if
\begin{displaymath}
  S(x-y,x-y)  \geq 0, \quad x, y
  \in G,   
\end{displaymath}
or, equivalently, if its projection multi-function has the natural
fixed-point property:
\begin{displaymath}
  x\in P_G(x), \quad x\in G. 
\end{displaymath}
We denote by $\mathcal{M}(S)$ the family of closed non-empty
$S$-monotone sets in $\real{d}$.  We refer to \citet{Fitzpatrick:88},
\citet{Simons:07}, and~\citet{Penot:09} for the results on Fitzpatrick
functions $\psi_G$ associated with $G \in \mathcal{M}(S)$.

\begin{Example}[Standard form]
  \label{ex:1}
  If $d=2m$ and
  \begin{displaymath}
    S(x,y) = \sum_{i=1}^m \cbraces{x^i y^{m+i} + x^{m+i}y^i},  \quad
    x,y\in \real{2m}, 
  \end{displaymath}
  then $S \in \msym{2m}m$ and the $S$-monotonicity means the standard
  monotonicity in $\real{2m}=\real{m}\times \real{m}$. For a
  \emph{maximal} monotone set $G$, the function $\psi_G$ becomes the
  classical Fitzpatrick function from~\cite{Fitzpatrick:88}.
\end{Example}

\begin{Example}[Canonical form]
  \label{ex:2}
  If $\Lambda$ is the \emph{canonical} quadratic form in $\msym{d}m$:
  \begin{displaymath}
    \Lambda(x,y) = \sum_{i=1}^m x^i y^i  - \sum_{i=m+1}^d x^i y^i, \quad
    x,y\in \real{d},  
  \end{displaymath}
  then a closed set $G$ is $\Lambda$-monotone if and only if
  \begin{displaymath}
    G = \graph{f} \set \descr{(u,f(u))}{u\in D}, 
  \end{displaymath}
  where $D$ is a closed set in $\real{m}$ and $\map{f}{D}{\real{d-m}}$
  is a $1$-Lipschitz function:
  \begin{displaymath}
    \abs{f(u)-f(v)} \leq \abs{u-v}, \quad u,v \in D.  
  \end{displaymath}
  In view of the Kirszbraun Theorem, \cite[2.10.43]{Feder:69}, $G$ is
  maximal $\Lambda$-monotone if and only if $D=\real{m}$.
\end{Example}

While working on the $S$-space, it is convenient to use appropriate
versions of subdifferential and Clark-type subdifferential for a
convex function $f$. For $x\in \cl{\dom{f}}$, they are defined as
\begin{align*}
  \partial^Sf(x) \set & \descr{y\in \real d}{S(z,y)\leq f(x+z)-f(x), \,
                        z\in \real{d}} \\
  = & \descr{S^{-1}y\in \real d}{\ip{z}{y} \leq f(x+z)-f(x), \,
      z\in \real{d}} \\
  = & \; S^{-1} \partial f(x), \\
  \cpartial^Sf(x) \set & S^{-1}\cpartial f(x). 
\end{align*}

\begin{Lemma}
  \label{lem:6}
  Let $S\in \msym{d}{m}$, $G\in \mathcal{M}(S)$, and assume that
  $D\set \interior\dom{\psi_G} \not=\emptyset$.  Then
  \begin{gather*}
    \range{S^{-1}\nabla \psi_G} \subset G, \\
    \dom{P_G} = \dom{\partial^S \psi_G}  = \dom{\cpartial^S \psi_G},
  \end{gather*}
  and
  \begin{gather*}
    \psi_G(x) = \psi_G^*(Sx) = \frac12 S(x,x), \quad x\in G, \\
    P_G(x) = \partial^S \psi_G(x) \cap G, \quad x\in \real{d}.
  \end{gather*}
\end{Lemma}
\begin{proof}
  We write $\psi_G$ as
  \begin{displaymath}
    \psi_G(x) = g^*(x) \set \sup_{y\in \real{d}}(\ip{x}{y} - g(y)), \quad
    x\in \real{d}, 
  \end{displaymath}
  where $g(y)=\frac{1}{2}S^{-1}(y,y)$ for $y\in SG$ and
  $g(y) = +\infty$ for $y\not\in SG$.  From the definition of
  $\psi _G$ and the $S$-monotonicity of $G$ we deduce that
  \begin{displaymath}
    \psi _G(x)=\frac12 S(x,x) - \frac12 \inf_{y\in
      G} S(x-y,x-y) = \frac12 S(x,x), \quad x\in G, 
  \end{displaymath}
  and then that
  \begin{displaymath}
    \psi _G(S^{-1}x)  \leq g(x), \quad x\in
    \real{d}. 
  \end{displaymath}
  As $\psi^*_G=g^{**}$ and $g^{**}$ is the largest closed convex
  function less than $g$, we have that
  \begin{displaymath}
    \psi _G(S^{-1}x) \leq \psi ^*_G(x)\leq g(x), \quad x\in \real{d}.
  \end{displaymath}
  Putting together the relations above, we obtain that
  \begin{displaymath}
    \frac 12 S(x,x)=\psi _G(x)\leq \psi ^*_G(Sx)\leq
    g(Sx)=\frac{1}{2}S(x,x), \quad x\in G.  
  \end{displaymath}
  For every $x\in \real{d}$, the values of the Fitzpatrick function
  $\psi_G$ and of the projection multi-function $P_G$ can now be
  written as
  \begin{align*}
    \psi_G(x) & = \sup_{y\in G} \cbraces{S(x,y) - \psi_G^*(Sy)} =
                \sup_{z\in SG} \cbraces{\ip{x}{z} - \psi_G^*(z)}, \\ 
    P_G(x) & = \argmax_{y\in G}\cbraces{S(x,y)-\frac 12S(y,y)} = S^{-1}
             \argmax_{z\in SG}\cbraces{\ip{x}{z}-\psi^*_G(z)}.
  \end{align*}
  The stated relations between $P_G$ and $\partial^S \psi_G$ follow
  from Lemma~\ref{lem:5} as soon as we observe that
  $P_G(x)= \partial^S \psi_G(x)=\emptyset$ for $x\notin \cl D$. The
  equality of the domains of $\partial^S \psi_G$ and
  $\cpartial^S \psi_G$ is just a restatement of~\eqref{eq:7}.
  Finally, Lemma~\ref{lem:4} yields the inclusion of
  $\range \nabla \psi_G$ into $SG$.
\end{proof}

To facilitate geometric interpretations, we also adapt the concept of
a normal vector to the product structure of the $S$-space.  Let $H$ be
a closed set in $\real{d}$.  A vector $w\in \real{d}$ is called
$S$-\emph{regular normal to $H$ at $x\in H$} if
\begin{displaymath}
  \limsup_{\substack{H\ni y\to x \\ y\not= x}}
  \frac{S(w,y-x)}{\abs{y-x}} =
  \limsup_{\substack{H\ni y\to x \\y\not=
      x}} \frac{\ip{Sw}{y-x}}{\abs{y-x}} \leq 0.   
\end{displaymath}
A vector $w\in \real{d}$ is called $S$-\emph{normal to $H$ at $x$} if
there exist $x_n\in H$ and an $S$-regular normal vector $w_n$ to $H$
at $x_n$, such that $x_n\rightarrow x$ and $w_n\rightarrow w$.  In
other words, $w$ is $S$-(regular) normal to $H$ at $x\in H$ if $Sw$ is
(regular) normal to $H$ at $x$ in the classical Euclidean sense.  It
is easy to see that if $x\in P_H(z)$, then the vector $z-x$ is
$S$-regular normal to $H$ at $x$.

\begin{Lemma}
  \label{lem:7}
  Let $S\in \msym{d}{m}$, $j\in \braces{1,\dots,d}$,
  $C\in \mathcal{C}^j$, $h$ be given by~\eqref{eq:1} for convex
  functions $g_1$ and $g_2$ with linear growth, and $H$ be the
  zero-level set of the composition function $h\circ S$:
  \begin{displaymath}
    H \set \descr{x\in \real{d}}{h(Sx) = 0}.
  \end{displaymath}
  Then $h\in \mathcal{H}^j_C$ if and only if for every $x\in H$, there
  exists an $S$-normal vector $w\in C$ to $H$ at $x$.
\end{Lemma}

\begin{proof}
  Let $z\in H$ and $w\in \real{d}$. Setting
  \begin{displaymath}
    SH \set \descr{Sx}{x\in H} = \descr{x\in \real{d}}{h(x) = 0}, 
  \end{displaymath}
  recalling that $S(w,y-z) = \ip{w}{Sy-Sz}$, and using the trivial
  inequalities:
  \begin{displaymath}
    \frac1{\norm{S^{-1}}} \abs{y-z} \leq
    \abs{Sy-Sz} \leq \norm{S} \abs{y-z}, \quad y\in \real{d}, 
  \end{displaymath}
  where $\norm{A} \set \max_{\abs{x}=1} \abs{Ax}$ for a $d\times d$
  matrix $A$, we deduce that $w$ is normal to $SH$ at $Sz$ if and only
  if $w$ is $S$-normal to $H$ at $z$. Lemma~\ref{lem:1} yields the
  result.
\end{proof}

Theorem~\ref{th:3} and Lemma~\ref{lem:7} show that the singular sets
$\Sigma^j(P_G)$ and $\Sigma^j (\cpartial^S\psi_G)$ can be covered by
countable $c-c$ surfaces that have at each point an $S$-normal vector
$w \in G-G$, with $w^j=1$. By the $S$-monotonicity of $G$, such vector
$w$ points to the \emph{non-negative} direction in the $S$-space:
$S(w,w) \geq 0$.

\begin{Theorem}
  \label{th:3}
  Let $S\in \msym{d}{m}$, $G\in \mathcal{M}(S)$,
  $j\in \braces{1,\dots, d}$, and assume that
  $D\set \interior{\dom{\psi_G}} \not=\emptyset$.  For every
  $n\geq 1$, there exist a compact set $C_n \subset G-G$ and a
  function $h_n \in \mathcal{H}^{j}_{\theta^{j}(C_n)}$ such that
  \begin{displaymath}
    \Sigma^j(\cpartial^S \psi_G) \subset \Sigma^j(P_G)  \subset\bigcup_n \descr{x\in 
      \real{d}}{h_n(Sx)=0}.
  \end{displaymath}
  For any $\epsilon>0$, all $C_n$ can be chosen such that
  $\diam{\theta^j(C_n)}< \epsilon$.
\end{Theorem}

\begin{proof}
  Let $g(x) \set \psi_G(S^{-1}x)$, $x\in \real{d}$. Clearly,
  \begin{displaymath}
    \cpartial^S \psi_G(x) \set S^{-1} \cpartial
    \psi_G(x) = \cpartial g(Sx), \quad x\in \real{d}. 
  \end{displaymath}
  Lemma~\ref{lem:6} shows that
  $S^{-1} \range \nabla \psi _G \subset G$ and
  \begin{displaymath}
    P_G(x) = \partial^S \psi_G(x) \cap G = \cbraces{S^{-1} \partial
      \psi_G(x)} \cap G = \partial g(Sx) \cap G, \quad x\in \real{d}.
  \end{displaymath}
  It follows that
  \begin{align*}
    \Sigma^j(\cpartial^S \psi_G) = S^{-1} \Sigma^j(\cpartial g), \quad
    \Sigma^j(P_G) = S^{-1} \Sigma^j_G(\partial g),  
  \end{align*}
  and a direct application of Theorem~\ref{th:2} yields the result.
\end{proof}

A set $A\subset \real{d}$ is called $S$-\emph{isotropic} if
\begin{displaymath}
  S(x-y,x-y) = 0, \quad x,y \in A. 
\end{displaymath}
We denote by $\mathcal{I}(S)$ the family of all closed $S$-isotropic
subsets of $\real{d}$.

Motivated by the study of existence and uniqueness of backward
martingale transport maps in~\cite{KramSirb:23}, we decompose the
singular set of $P_G$ as
\begin{align*}
  \Sigma(P_G) &\set \descr{x\in \dom{P_G}}{P_G(x) \text{~is not a
                point}} = \Sigma_0(P_G) \cup
                \Sigma_1(P_G), \\
  \Sigma_0(P_G) & \set \descr{x\in \Sigma(P_G)}{P_G(x) \in
                  \mathcal{I}(S)}, \\
  \Sigma_1(P_G) & \set \descr{x\in \Sigma(P_G)}{S(y_1-y_2,y_1-y_2)>0
                  \text{~for \emph{some}~} y_1,y_2 \in P_G(x)}. 
\end{align*}
We further write $\Sigma_1(P_G)$ as
\begin{displaymath}
  \Sigma_1(P_G) = \bigcup_{j=1}^d \Sigma_1^j(P_G), 
\end{displaymath}
where $\Sigma_1^j(P_G)$ consists of $x\in \Sigma(P_G)$ such that
$S(y_1-y_2,y_1-y_2)>0$ for some $y_1,y_2 \in P_G(x)$ with
$y_1^j \not = y_2^j$.

Theorem~\ref{th:4} and Lemma~\ref{lem:7} show that $\Sigma^j_1(P_G)$
can be covered by countable $c-c$ surfaces that have at each point a
strictly $S$-positive $S$-normal vector $w$ with $w^j=1$.

\begin{Theorem}
  \label{th:4}
  Let $S\in \msym{d}{m}$, $G\in \mathcal{M}(S)$,
  $j\in \braces{1,\dots,d}$, and assume that
  $D\set \interior{\dom{\psi_G}} \not=\emptyset$. For every $n\geq 1$,
  there exist a compact set $C_n \subset G-G$ with
  \begin{displaymath}
    x^j>0 \text{~and~} S(x,x) > 0, \quad x\in C_n, 
  \end{displaymath}
  and a function $h_n \in \mathcal{H}^{j}_{\theta^j(C_n)}$ such that
  \begin{displaymath}
    \Sigma^j_1(P_G)  \subset\bigcup_n \descr{x\in 
      \real{d}}{h_n(Sx)=0}.
  \end{displaymath}
  For any $\epsilon>0$, all $C_n$ can be chosen such that
  $\diam{\theta^j(C_n)}< \epsilon$.
\end{Theorem}

\begin{proof}
  Fix $\epsilon >0$.  Let $(x_n)$ be a dense sequence in $G$ and
  $(r_n)$ be an enumeration of all positive rationals.  Denote by
  $\alpha\set (m, n,k,l)$ the indexes for which the compact sets
  \begin{displaymath}
    C^{\alpha}_1 \set\descr{x\in G }{ |x-x_m|\leq r_k}, \quad
    C^{\alpha}_2 \set\descr{x\in G}{|x-x_n|\leq r_l},
  \end{displaymath}
  satisfy the constraints:
  \begin{displaymath}
    \diam  \theta^j(C^{\alpha}_2-C^{\alpha}_1) < \epsilon
    \text{~and~} x^j > 0, \; S(x,x) > 0, \; x \in C_2^{\alpha} -
    C_1^{\alpha}. 
  \end{displaymath}
  We have that
  \begin{displaymath}
    \Sigma^j_1(P_G)  =
    \bigcup_{\alpha}\Sigma_{C^{\alpha}_1, C^{\alpha}_2}(P_G)=
    \bigcup_{\alpha} \descr{x\in \real{d}}{P_G(x)\cap
      C^\alpha_i\not=\emptyset, \; i=1,2}.
  \end{displaymath}
  Let, again,  $g(x) \set \psi_G(S^{-1}x)$, $x\in \real{d}$. From
  Lemma~\ref{lem:6} we deduce that
  \begin{displaymath}
    P_G(x) = \cbraces{S^{-1} \partial \psi_G(x)} \cap G = \partial
    g(Sx) \cap G, \quad x\in \real{d}. 
  \end{displaymath}
  It follows that
  \begin{displaymath}
    x\in \Sigma_{C^{\alpha}_1, C^{\alpha}_2}(P_G) \iff
    Sx \in  \Sigma_{C^{\alpha}_1, C^{\alpha}_2}(\partial g). 
  \end{displaymath}
  Applying Theorem~\ref{th:1} to each singular set
  $\Sigma_{C^{\alpha}_1, C^{\alpha}_2}(\partial g)$, we obtain the
  result.
\end{proof}

The singular set $\Sigma_0(P_G)$ is included into a larger set
\begin{displaymath}
  \overline{\Sigma}_0(P_G) = \bigcup_{j=1}^d
  \overline{\Sigma}^j_0(P_G),  
\end{displaymath}
where $\overline{\Sigma}^j_0(P_G)$ consists of $x\in \dom{P_G}$ such
that $S(y_1-y_2,y_1-y_2)=0$ for \emph{some} $y_1,y_2 \in P_G(x)$ with
$y_1^j \not = y_2^j$.

Theorem~\ref{th:5} and Lemma~\ref{lem:7} show that
$\overline{\Sigma}_0^j(P_G)$ can be covered, for any $\delta >0$, by
countable $c-c$ surfaces that have at each point an $S$-normal vector
$w$ such that $w^j=1$ and $0\leq S(w,w)\leq \delta$.

\begin{Theorem}
  \label{th:5}
  Let $S\in \msym{d}{m}$, $G\in \mathcal{M}(S)$,
  $j\in \braces{1,\dots,d}$, and assume that
  $D\set \interior{\dom{\psi_G}} \not=\emptyset$. Let $\delta>0$.  For
  every $n\geq 1$, there exist a compact set $C_{n} \subset G-G$ with
  \begin{displaymath}
    x^j>0, \; x\in C_n, \text{~and~} 0\leq S(x,x) \leq
    \delta, \; x\in  \theta^j(C_{n}),   
  \end{displaymath}
  and a function $h_{n} \in \mathcal{H}^{j}_{\theta^j(C_{n})}$ such
  that
  \begin{displaymath}
    \overline{\Sigma}^j_0(P_G)  \subset \bigcup_n \descr{x\in 
      \real{d}}{h_{n}(Sx)=0}.
  \end{displaymath}
  For any $\epsilon>0$, all $C_{n}$ can be chosen such that
  $\diam{\theta^j(C_n)}< \epsilon$.
\end{Theorem}

\begin{proof}
  The proof is almost identical to that of Theorem~\ref{th:4}. We fix
  $\delta, \epsilon >0$ and find a \emph{countable} family
  $(C^\alpha_1,C^\alpha_2)$ of pairs of compact sets
  $C^\alpha_i \subset G$, $i=1,2$, such that
  \begin{gather*}
    x^j > 0, \; x \in C_2^{\alpha} - C_1^{\alpha}, \text{~and~}
    \diam  \theta^j(C^{\alpha}_2-C^{\alpha}_1) < \epsilon,\\
    0\leq S(x,x) < \delta, \; x \in \theta ^j(C_2^{\alpha} -
    C_1^{\alpha}),
  \end{gather*}
  and every pair $(y_1,y_2)$ of elements of $G$ with
  $S(y_1-y_2,y_1-y_2)=0$ and $y_1^j \not = y_2^j$ is contained in some
  $(C^\alpha_1,C^\alpha_2)$. Then
  \begin{displaymath}
    \overline{\Sigma}^j_0(P_G)  \subset
    \bigcup_{\alpha}\Sigma_{C^{\alpha}_1, C^{\alpha}_2}(P_G)=
    \bigcup_{\alpha} S^{-1} \Sigma_{C^{\alpha}_1, C^{\alpha}_2}(\partial g), 
  \end{displaymath}
  where $g(x) \set \psi_G(S^{-1}x)$, $x\in \real{d}$, and
  Theorem~\ref{th:1} applied to
  $\Sigma_{C^{\alpha}_1, C^{\alpha}_2}(\partial g)$ yields the result.
\end{proof}

The geometric description of the zero order singularities becomes
especially simple if the index $m=1$. In this case,
$\overline{\Sigma}^j_0(P_G)$ can be covered by countable number of
hyperplanes whose $S$-normal vectors are $S$-isotropic.

\begin{Theorem}
  \label{th:6}
  Let $S\in \msym{d}{1}$, $G\in \mathcal{M}(S)$,
  $j\in \braces{1,\dots,d}$, and assume that
  $D\set \interior{\dom{\psi_G}} \not=\emptyset$. For every $n\geq 1$,
  there exist $y_n,z_n \in G$ such that
  \begin{equation}
    \label{eq:14}
    y^j_n-z^j_n >0, \quad S(y_n-z_n,y_n-z_n) = 0, 
  \end{equation}
  and
  \begin{displaymath}
    \overline{\Sigma}^j_0(P_G) \subset\bigcup_n
    \descr{x\in \real{d}}{S(x-z_n,y_n-z_n)= 0}. 
  \end{displaymath}
\end{Theorem}

\begin{proof}
  By the law of inertia for quadratic forms, \cite[Theorem~1,
  p.~297]{Gant-1:98}, there exists a $d\times d$ matrix $V$ of full
  rank such that
  \begin{displaymath}
    S = V^T \Lambda V, 
  \end{displaymath}
  where $V^T$ is the transpose of $V$ and $\Lambda$ is the diagonal
  matrix with diagonal entries $\braces{1,-1,\dots,-1}$. In other
  words, $\Lambda$ is the canonical quadratic form for
  $\mathcal{S}^d_1$ from Example~\ref{ex:2}.  As
  $S(x,x) = \Lambda(Vx,Vx)$, we deduce that
  $F\set VG \in \mathcal{M}(\Lambda)$ and
  \begin{displaymath}
    A\in \mathcal{I}(S) \iff VA \in
    \mathcal{I}(\Lambda).
  \end{displaymath}
  As we pointed out in Example~\ref{ex:2},
  \begin{displaymath}
    F = \graph{f} = \descr{(t,f(t))}{t\in D},
  \end{displaymath}
  for a $1$-Lipschitz function $\map{f}{D}{\real{d-1}}$ defined a
  closed set $D\subset \real{}$.

  Let $x=(s,f(s))$ and $y=(t,f(t))$ be distinct elements of $F$, where
  $s,t\in D$. We have that $\braces{x,y} \in \mathcal{I}(\Lambda)$ if
  and only if $\abs{f(t)-f(s)} = \abs{t-s}$. The $1$-Lipschitz
  property of $f$ then implies that
  \begin{displaymath}
    f(r) = f(s) + \frac{r-s}{t-s} \cbraces{f(t)-f(s)}, \quad r\in D
    \cap [s,t].  
  \end{displaymath}
  It follows that the collection of all $\Lambda$-isotropic subsets of
  $F$ can be decomposed into an intersection of $F$ with at most
  countable union of line segments whose relative interiors do not
  intersect.

  The same property then also holds for the $S$-isotropic subsets of
  $G$. Thus, there exist $y_n,z_n \in G$, $n\geq 1$,
  satisfying~\eqref{eq:14} and such that every $S$-isotropic subset of
  $G$ having elements with distinct $j$th coordinates is a subset of
  some $S$-isotropic line
  $L_n \set \descr{y_n + t(z_n-y_n)}{t\in \real{}}$. In particular, if
  $x\in \overline{\Sigma}_0^j(P_G)$, then $P_G(x)$ intersects some
  line $L_n$ at distinct $y$ and $z$. We have that
  \begin{displaymath}
    2S(x-z,y-z)  = S(x-z,x-z) + S(y-z,y-z) - S(x-y,x-y) = 0. 
  \end{displaymath}
  As $y,z\in L_n\in \mathcal{I}(S)$, we obtain that
  $S(x-z_n,y_n-z_n) = 0$, as required.
\end{proof}

\begin{Example}
  \label{ex:3}
  Let $d=2$ and $S$ be the standard bilinear form from
  Example~\ref{ex:1}:
  \begin{displaymath}
    S(x,y)=S((x_1,x_2),(y_1,y_2)) = x_1y_2 + x_2y_1.
  \end{displaymath}
  Let $G\in \mathcal{M}(S)$. As $S(x,x) = 2x_1x_2$, we have that
  \begin{displaymath}
    \Sigma_1(P_G) = \Sigma^j_1(P_G), \quad j=1,2. 
  \end{displaymath}
  Theorem \ref{th:4} yields convex functions $g_{1,n}$ and $g_{2,n}$
  on $\real{}$ and constants $\epsilon_n>0$, $n\geq 1$, such that
  ($g'=g'(t)$ is the derivative of $g=g(t)$)
  \begin{gather*}
    \epsilon_n \leq g'_{1,n}(t) - g'_{2,n}(t) \leq \epsilon_n^{-1},
    \quad t\in
    \dom{g'_{1,n}} \cap \dom{g'_{2,n}}, \\
    \Sigma _1 (P_G)\subset \bigcup _n \descr{x\in \real{2}}{x_2=
      g_{2,n}(x_1) - g_{1,n}(x_1)}.
  \end{gather*}
  Theorem \ref{th:6} yields sequences $(x^n_1)$ and $(x^n_2)$ in
  $\real{}$ such that
  \begin{displaymath}
    \overline{\Sigma}^1_0(P_G)\subset
    \bigcup _n \descr{x\in \real{2}}{x_2=x^n_2}, \quad
    \overline{\Sigma}^2_0(P_G)\subset 
    \bigcup _n \descr{x\in \real{2}}{x_1=x^n_1}.
  \end{displaymath}
  These results improve \cite[Theorem~B.12]{KramXu:22}, where $G$ is a
  maximal monotone set and $g_n \set g_{1,n} - g_{2,n}$ is a strictly
  increasing Lipschitz function such that
  $\epsilon_n \leq g'_n(t) \leq \epsilon^{-1}_n$, whenever it is
  differentiable.
\end{Example}

\section*{Acknowledgments}

We thank Giovanni Leoni for pointing out the
references~\cite{Zahorski:46} and \cite{FowlPreiss:09}. We also thank
an anonymous referee for constructive remarks and for bringing to our
attention the references~\cite{BenLin:00} and~\cite{Thib:23}.

\bibliographystyle{plainnat} 
\def\cprime{$'$}

\end{document}